\documentclass[11pt,reqno,oneside]{amsart}
\usepackage[english]{babel}
\usepackage[utf8]{inputenc}
\usepackage[mathscr]{eucal}
\usepackage{color}
\usepackage{xcolor}
\usepackage{amsfonts}   
\usepackage{amssymb}
\usepackage{latexsym}
\usepackage{enumerate}
\usepackage{mathtools}

\usepackage[colorlinks,citecolor=purple,linkcolor=blue]{hyperref}

\usepackage{geometry}

\newtheorem{theorem}{Theorem}[section]
\newtheorem{corollary}[theorem]{Corollary}
\newtheorem{lemma}[theorem]{Lemma}
\newtheorem{proposition}[theorem]{Proposition}
\theoremstyle{definition}
\newtheorem{definition}[theorem]{Definition}
\newtheorem{example}[theorem]{Example}

\theoremstyle{remark}
\newtheorem{remark}[theorem]{Remark}

\def\k{\mathsf k}
\def\C{\mathbb C}

\newcommand{\Frac}{{\rm{Frac}}}
\newcommand{\Quot}{{\rm{Quot}}}

\newcommand{\Autk}{{\rm{Aut}}}

\newcommand{\id}{{\rm{Id}}}

\newcommand{\into}{\,\,\hookrightarrow\,\,}

\newcommand{\onto}{\,\,\twoheadrightarrow\,\,}

\def\K{\mathcal K}
\def\L{\mathcal L}

\def\M{\mathcal M}

\def\Oo{\mathcal O}
\def\Z{\mathbb Z}

\def\N{\mathbb N}

\def\A{\mathcal A}
\def\AAA{\mathfrak A}

\def\XX{\mathfrak {X}}
\def\H{\mathbb H}

\def\cfs{\operatorname{cfs}}
\def\m{\mathsf m}
\def\n{\mathsf n}
\def\I{\mathcal I}
\def\Mm{\mathfrak M}
\def\Nn{\mathfrak N}

\def\supp{\mathsf{supp}}

\def\Supp{\operatorname{Supp}}

\newcommand{\Za}{\mathbb{Z}^{\oplus \alpha}}

\newcommand{\SpecMax}{\operatorname{SpecMax}}

\def\Irred{\mathsf{Irred}}

\makeatletter

\newcommand{\colim@}[2]{%
  \vtop{\m@th\ialign{##\cr
    \hfil$#1\operator@font colim$\hfil\cr
    \noalign{\nointerlineskip\kern1.5\ex@}#2\cr
    \noalign{\nointerlineskip\kern-\ex@}\cr}}%
}
\newcommand{\colim}{%
  \mathop{\mathpalette\colim@{\rightarrowfill@\scriptscriptstyle}}\nmlimits@
}
\renewcommand{\varprojlim}{%
  \mathop{\mathpalette\varlim@{\leftarrowfill@\scriptscriptstyle}}\nmlimits@
}
\renewcommand{\varinjlim}{%
  \mathop{\mathpalette\varlim@{\rightarrowfill@\scriptscriptstyle}}\nmlimits@
}
\makeatother

\newenvironment{dedication}
  {
   \itshape             
   \raggedleft          
  }

\begin{document}

\title{Characteristic free Galois rings and generalized Weyl algebras}

\author{Jo\~ao Schwarz}
\address{Shenzhen International Center for Mathematics, Southern University of Science and Technology, Shenzhen, China}
\email{jfschwarz.0791@gmail.com}


\maketitle

\begin{dedication}
    Dedicated to Fl\'avia and Lais.
\end{dedication}

\medskip

\medskip

\begin{abstract}

This paper develops from scratch a theory of Galois rings and orders over arbitrary fields. Our approach is different from others in the literature in that there is no non-modularity assumption. We prove, when the field is algebraically closed, the analogue of the Main Theorem of the representation theory of Galois orders by V. Futorny and S. Ovsienko. Then we develop a theory of infinite rank generalized Weyl algebras, which was never explicitly introduced in the literature before, and prove its basic properties. We expect their representation theory to be of interest for future works. Finally we show that under very mild assumptions, the invariants of generalized Weyl algebras under the action of non-exceptional irreducible complex reflection groups are a principal Galois orders, greatly generalizing, in an elementary fashion, results obtained previously for the Weyl algebras,

\end{abstract}

\section{Introduction}

An abstract framework that unifies the representation theory of classical Harish-Chandra modules \cite[Chapter 9]{Dixmier} \cite[Chapter 6]{Jantzen}, Gelfand-Tsetlin modules for $\mathfrak{gl}_n$ \cite{DFO0}, and generalized weight modules for generalized Weyl algebras \cite{Bavula} was introduced in \cite{DFO}. In that paper, an abstract notion of Harish-Chandra modules of an associative algebra $A$ with respect to a certain subalgebra $\Gamma$ that is a \textit{Harish-Chandra subalgebra} was introduced. For further developments, see \cite{Fillmore}, \cite{FillHart}, \cite{Schwarz2}.

An offspring of this theory is in the case where $\Gamma$ is commutative and $A$ embeds in a suitable skew (monoid) group ring after localization at $\Gamma^\times$. This is the theory of Galois rings and orders, developed by V. Futorny and S. Ovsienko in the papers \cite{FO} and \cite{FO2}. Significant contributions to this theory are the works \cite{Hartwig}, which introduces the notion of principal Galois order and canonical Harish-Chandra modules; \cite{Webster}, which introduces the notion of a flag order a shows a deep connection of the theory with supersymmetric quantum field theories; and \cite{Hartwig2}, that generalizes the original setting with the role of a skew group ring being replaced by the smash product with a Hopf algebra, and arbitrary base fields.\footnote{In the original setting of \cite{FO} \cite{FO2} the base field is assumed to be algebraically closed and of zero characteristic.} Further developments of both structure and representation theory can be found in \cite{EFOS}, \cite{FS3}, \cite{Hartwig3}, \cite{Jauch}, \cite{Jauch2}, \cite{MV}, \cite{FGRZ}, \cite{lots}, \cite{FHJS}.

In \cite{Hartwig2} it was made manifest the possibility of a good representation theory for Galois rings defined over fields of prime characteristic. In particular, the Main Theorem of the theory of Galois order (cf. \cite[Main Theorem]{FO2}) holds for \textbf{principal} Galois orders in what we call the non-modular case \cite[Corollary 5.7]{Hartwig2}. The main point of this paper is to drop the non-modularity assumption and generalize the results of \cite{Hartwig2} to the \textit{modular case} and Galois orders in the original sense \cite{FO2}. We remark that such generalizations from the point of view of structure theory of rings were carried out in \cite{FHJS}.

After some preliminaries in Section 2, in Section 3 we discuss what are called in the notation and terminology of \cite{Hartwig2} spherical Galois orders when $\tilde{\mathcal{H}}$ is a monoid ring $\k \M$. We develop the theory from scratch, following \cite{FO} and \cite{Hartwig}, with some proofs omitted. Our main task is to show that results from \cite{Hartwig2} also hold in the modular case (cf. Definition \ref{main-def-3}). We have simplified the exposition in the literature in order to make clear that we can work over any field. Our main result in this section is a broad generalization of the fact proved in \cite{Jauch2} that the tensor product of two principal Galois orders over an algebraically closed field of zero characteristic, in the setting of \cite{Hartwig}, is again a principal Galois order. Our generalization is Theorem \ref{main-section-3}.

Section 4 is dedicated to the proof of the Main Theorem of Galois order theory for modular, not necessarily principal, Galois orders, generalizing \cite[Corollary 5.7]{Hartwig2}, Theorem \ref{Main-Theorem}. We follow \cite{FO} \cite{FO2}.

In Section 5 we settled an important open problem in Galois order theory (see \cite{Schwarz2}), and obtained a simple example of a Galois order that is not principal. As was noted in \cite{FS} and \cite{FSS}, most Galois rings in $(L*\M)^W$ have $L$ and $K$ rational and $\M \simeq \Z^n$. We give an example based on the theory of elliptic curves where none of this holds (Proposition \ref{elliptic}), and we generalize the notion of rational Galois order from \cite{Hartwig} to the so called notion of linear algebraic Galois ring (Definition \ref{linear-algebraic-Galois-ring}); we recover Hartwig's notion when the linear algebraic group is a finite dimensional vector space. We show that the Galois rings considered in \cite{FH} and \cite{FS2} are examples of this new notion when the linear algebraic group is a torus.

In Section 6 we develop the notion, introduced in \cite{Schwarz2}, of infinite degree generalized Weyl algebras. The difference is that in this paper we deal with ordinal and not cardinal numbers as degrees of the GWAs. Our main technical tool is Theorem \ref{GWA-as-limits}. With this, we obtain a simplicity criteria for infinite rank degree generalized Weyl algebras in Theorem \ref{GWA-simple}, and for them to be Ore domains (Corollary \ref{GWA-Ore}). We show an embedding in a skew group ring, reminiscent of \cite[Proposition 13]{FS3}, in Theorem \ref{GWA-pre-Galois-ring}. However, since we are considering infinite degree GWAs and noncommutative base rings, the proof is much more difficult (cf. Theorem \ref{GWA-localization}). We finish this section with a birational classification of degree $1$ generalized Weyl algebras over the polynomial ring in one variable (Corollary \ref{birational}).

In Section 7 we show that fixed rings of generalized Weyl algebras when the base ring is an affine domain over an algebraically closed field of zero characteristic under the action of the Shephard-Todd groups $G(n,p,m)$ are principal Galois orders, and we obtain an analogue of noncommutative Noether's problem \cite{AD} and its q-deformation \cite{FH} for classical and quantum generalized Weyl algebras. Our exposition is an improved version of that found in \cite{Schwarz2}. Freeness results over the Harish-Chandra subalgebra are also discussed.




\section{Preliminaries}

In this section we collect the basic facts about the abstract theory of Harish-Chandra subalgebras and modules from \cite{DFO}, and we prove some simple results (which are probably well known) about inductive limits of rings, together with some applications.

\subsection{Reminder of the DFO setting}

In this subsection we recall the basic definitions from \cite{DFO}, further studied in \cite{Fillmore}, \cite{FillHart}, \cite{Schwarz2}.

Let $U$ be an associative algebra and $\Gamma$ a subalgebra. We denote by $\cfs \, \Gamma$ its cofinite spectrum: that is, the set of maximal ideals $\m$ such that $\Gamma/\m$ is finite dimensional. If $\m$ is such an ideal, then $\Gamma/\m$ , by Wedderburn-Artin Theory, is a matrix ring over a finite dimensional division $\k$-algebra, and hence has a unique simple module $\mathsf{S}_\m$.

In \cite{DFO}, two essential properties of $\Gamma$ were considered: being \emph{quasi-central} and \emph{quasi-commutative}.

\begin{definition}\cite{DFO}
    $\Gamma$ is called quasi-central in $U$ if, for every $u \in U$, the $\Gamma$-bimodule $\Gamma u \Gamma$ is a finitely generated left and right $\Gamma$-module.
\end{definition}

\begin{definition}\cite{DFO}
$\Gamma$ is quasi-commutative if, given distinct $\m, \n \in \cfs(\Gamma)$, $\operatorname{Ext}^1(\mathsf{S}_\m, \mathsf{S}_\n)=0$.
\end{definition}

\begin{definition}\cite{DFO}
If we have an alebra $U$ and a subalgebra $\Gamma$ that is both quasi-central and quasi-commutative, then $\Gamma$ is called a \textit{Harish-Chandra subalgebra}.   
\end{definition}

Harish-Chandra subalgebras are crucial for representation theory \cite{DFO} \cite{Fillmore} \cite{FillHart} \cite{FO2} \cite{Hartwig2}, \cite{Webster}, \cite{lots}, \cite{Schwarz2}. In the framework of Galois rings, since $\Gamma$ is always commutative, the only thing remaining to check is being quasi-central.

To verify that an algebra is quasi-central, we have the following useful general results.

\begin{theorem}\label{theorem-quasi-central}
    Let $\Gamma \subset \Lambda$ be two associative algebras, with $\Gamma$ Noetherian and $\Lambda$ a finitely generated left and righ $\Gamma$-module. Let $A$ be an associative algebra containing a localization of $\Lambda$, called $L$; and generated, as an algebra, by $L$ and $\mathsf{X}$, where each $X \in \mathsf{X}$ is such that $X \Gamma \, \, (\Gamma X) \subset \Lambda X \, \, (X \Lambda)$. Finally, assume that $\Lambda \Gamma =\Gamma \Lambda$ in $A$. Then $\Gamma$ is a quasi-central subalgebra of $A$.
\end{theorem}
\begin{proof}
    For any $a,b \in A$, $\Gamma(a+b)\Gamma=\Gamma a \Gamma + \Gamma b \Gamma$, and $\Gamma a b \Gamma \subset (\Gamma a \Gamma) (\Gamma b\Gamma)$. Hence, as $\Gamma$ is Noetherian, it is enough to show that $\Gamma a \Gamma$ is a finitely generated left and right $\Gamma$ module for $a \in L \cup \mathsf{X}$. If $a \in L$, as $\Gamma \Lambda = \Lambda \Gamma$ and $L=\Lambda_{S^{-1}} \, (_{S^{-1}}\Lambda)$, $\Gamma a \Gamma=\Gamma a=a\Gamma$, and so we are done. If $a=X \in \mathsf{X}$, $\Gamma X \Gamma \subset \Lambda X$. As $\Gamma$ is Noetherian and $\Lambda$ is a finitely generated left $\Gamma$-module, so is $\Lambda X$, for it is an homomorphic image of $\Lambda$. Again, since $\Gamma$ is Noetherian and $\Gamma X \Gamma$ is a submodule of $\Lambda X$, it is as well a finitely generated left $\Gamma$-module. We have by a similar reasoning that $\Gamma X \Gamma$ is a finitely generated right $\Gamma$-module. Hence $\Gamma$ is quasi-central.
\end{proof}

\begin{proposition}\label{proposition-quasi-central}
Let $\Gamma$ be a subalgebra of an algebra $A$, and suppose that $A$ is generated by a family of elements $\{a_i \}_{i \in I}$ with the property that for each $a_i$, $\Gamma a_i = a_i \Gamma$. Then $\Gamma$ is a quasi-central subalgebra.    
\end{proposition}
\begin{proof}
    Let $b_1b_2 \ldots b_k$ be a product of $a_i$'s. Then $\Gamma b_1b_2 \ldots b_k=b_1b_2 \ldots b_k \Gamma$. Also, if $x, y \in A$ are such that $\Gamma x=x \Gamma$ and $\Gamma y= y \Gamma$, then $\Gamma(x+y)=(x+y)\Gamma$. Hence we are done.
\end{proof}

\subsection{Inductive limit of rings}

In this subsection we collect some results about inductive limits of rings that are probably well known, but for which we could not find an appropriate reference to quote.

Recall that an (uppward) directed set is a poset $\mathcal{I}$ with the property that for each $i,j \in \I$ there is a $k \in I$ with $k \geq i,j$. A direct system of rings indexed by $\I$ is a collection of rings $R_i$ indexed by $i \in I$, and connecting morphisms $f_{ij}:R_i \rightarrow R_j$ for $j \geq i$, such that $f_{ii}=\id$ and for $k \geq j \geq i$, $f_{ik}= f_{jk}. \circ f_{ij}$. If each $f_{ij}$ is injective, we say that the directed system is \textit{injective}.

The ring $R=\varinjlim_{i\in\I} R_i$ has the concrete description as $\bigsqcup_{i \in \I} R_i/ \sim$, where $a \in R_i$ and $b \in R_j$ are related by $\sim$ if and only if there exists an $k \geq i,j$ with $f_{ki}(a)=f_{kj}(b)$; addition and multiplication are defined as usual (cf. \cite[Appendix A]{Matsumura}). We denote the equivalence class of some $a \in R_i$ as $[a]$.

For each $i \in \I$ and $a \in R_i$, consider the ring morphism $\phi_i:R_i \rightarrow R$, $a \mapsto [a]$. Then $\phi_i(r_i)=\phi_j(f_{ij}(r_i)), \, i,j \in \I, j \geq i, \, r_i \in R_i$. These are called the \emph{canonical morphisms}.

Finally, we have the following universal property: if we have a ring $B$ and for each $i \in \I$ morphisms $\psi_i: R_i \rightarrow B$ which for $j \geq i$, $\psi_i=\psi_j \circ f_{ij}$, then there is a unique morphism $\psi: R \rightarrow B$ with $\psi_i=\psi\circ\phi_i$ for each $i \in \I$.

\begin{proposition}\label{prop-inductive-limit}
    Consider an injective directed system of rings $R_i$ and its inductive limit $R=\varinjlim_{i\in\I} R_i$.
    \begin{enumerate}
        \item If every $R_i$ is a domain, then so is $R$.
        \item If every $R_i$ is a simple ring,then so is $R$.
        \item If every $R_i$ is an Ore domain, then so is $R$, and $\Frac \, R = \varinjlim_{i \in \I} \Frac \, R_i$
    \end{enumerate}
\end{proposition}
\begin{proof}
    For each $i \in \I$ and $a \in R_i$, consider the canonical morphism $\phi_i:R_i \rightarrow R$, $a \mapsto [a]$. We claim that $\phi_i$ is injective. Indeed, if we had $[a]=[0]$, for some $j \geq i$, we would have $f_{ij}(a)=0$; but $f_{ij}$ is injective, so this implies $a=0$. Identifiying $R_i$ with its image under $\phi_i$ in $R$, we can write $R=\bigcup_{i \in \I} R_i$\footnote{When we have an injective indutive limit of rings, the identification $R=\varinjlim_{i\in\I} R_i=\bigcup R_i$ will be made frequently and usually without comments.}. Assume now that each $R_i$ is a domain. Let $a, b$ belong to $R$. Then $a \in R_i$ and $b \in R_j$, and there is $k \geq i,j$, with $a,b \in R_k$. If $a,b \neq 0$ , then $ab \in R_k$ is also non-zero. So $R$ is a domain. For (2) assume each $R_i$ is a simple ring. If $I \lhd R$ is non-zero, we must show that $R=I$. As  $R=\bigcup_{i \in \I} R_i$, for some $j$, $I \cap R_j \neq (0)$. Since $R_j$ is simple, $I\cap R_j = R_j$. This implies $1 \in I$, and so $I=R$. Lastly, assume each $R_i$ is an Ore domain. By (1) we already have that $R$ is a domain. Let $s \in R^\times$ and $r \in R$. We must find $r' \in R$ and $s' \in R^\times$ such that $rs'=sr'$ --- and similarly for the left Ore condition. But there is an $j$ such that $r,s \in R_j$, so we can find the needed $r',s'$ already in $R_j$ by hypothesis. To show that  $\Frac \, R = \varinjlim_{i \in \I} \Frac \, R_i$, we first remark that the inverse limit of division rings is a division ring, by an argument similar to the proof of (1). Finally, it is clear that each $rs^{-1} \in \Frac \, R$ already belongs to some $\Frac \, R_j$ for a suitable $j \in \I$, and so we are done.
\end{proof}

We shall obtain as a consequence a version of the noncommutative Noether's problem \cite{AD}, in \textit{any} characteristic (cf. \cite{SchwarzPan}), for the infinite symmetric group.

\begin{definition}\label{infinite-Weyl}
    Let $W_n(\k)$ be the rank any Weyl algebra. Given the natural chain of embeddings $W_1(\k)\into W_2(\k) \into \ldots \into W_n(\k) \into \ldots$, define $W_\omega(\k)=\varinjlim W_n(\k)$.
\end{definition}

When $\k$ has zero characteristic, the representation theory of $W_\omega(\k)$ was considered in \cite{BBF} and \cite{FGM}.

\begin{theorem}
    Let $W_\omega(\k)=\k \langle x_1, x_1, \ldots, y_1, y_2, \ldots \rangle$ and let $S_\omega$ be the subgroup of all permutations of $\N^\times$ that fixes all but a finite number of elements. Then $\Frac \, W_\omega(\k)^{S_\omega}=\Frac \, W_\omega(\k)$, if $\operatorname{char} \k=0$ and the field is arbitrary, or if $\operatorname{char} \k=p>0$ and $\k$ is algebraically closed.
\end{theorem}
\begin{proof}
    By Proposition \ref{prop-inductive-limit}(3) and \cite[Theorem 1.1]{FS} in characteristic $0$ and \cite[Corollary 1.4]{SchwarzPan} for prime characteristic.
\end{proof}

\subsection{Somme noncommutative invariant theory}

We simply recall one important result from noncommutative invariant theory. It is a noncommutative analogue of Noehter's celebrated Theorem about invariants of affine commutative algebras.

\begin{theorem}\label{MS}
Let $R$ be a finitely generated  Noetherian $\k$-algebra. If $G$ is a finite group of automorphisms of $R$, with $|G|^{-1} \in \k$, then $R^G$ is a finitely generated $\k$-algebra, $R^G$ is Noetherian, and $R$ is a finitely generated $R^G$-module.
\end{theorem}
\begin{proof}
    The first claim is proven in \cite{MS}, the second one in \cite[Corollary 1.12]{Montgomery}, and the third one in \cite[Corollary 5.9]{Montgomery}.
\end{proof}



\section{Characteristic free Galois rings}

The first part of this paper consists in simplifying the setting of \cite{Hartwig2} for the case $\tilde{\mathcal{H}}$ is $\k \, \Autk \, \Lambda$ and $\mathcal{H}$ is $\k \M$, where $\M$ is a submonoid of $\Autk \, \Lambda$ (and $\Lambda$ is a Noetherian integral domain and $\k$-algebra).
A main distinction in our approach is that, unlike \cite{Hartwig2}, is not necessary to assume that the Galois ring is non-modular (cf. Definition \ref{main-def-3}).

\begin{definition}\label{main-def-1}\cite{Hartwig2}
    A \emph{flag ring setting} is given by a Noetherian $\k$-algebra $\Lambda$ which is an integral domain and $\M$ a submonoid of $\Autk_\k \, \Lambda$. We denote such a setting by $(\Lambda, \M)$. A \emph{Galois ring setting} is a tuple ($\Lambda$, $\M$, $W$), where $(\Lambda, \M)$ is a flag ring setting and $W$ is a finite subgroup of $\Autk_\k \, \Lambda$ acting by conjugation on $\M$ with a finite number of orbits and such that $\M$ is \emph{separating}, that is, $\M \M^{-1} \cap W = \{e \}$. If $|W|^{-1}$ belongs to $\k$, we say that the the Galois ring setting is \textit{non-modular}.
\end{definition}


Set $L=\Frac \, \Lambda$ and $\L=L*\M$ for the skew monoid ring $L \rtimes \k \M$. For $X \in \L, l \in L$, define $X(l):=\sum_{\mu \in \M} l_\mu \mu(l)$, if  $X=\sum_{\mu \in \M} l_\mu \mu$, $l_\mu \in L$ $\mu \in \M$.

\begin{definition}\label{main-def-2}\cite{Webster}
Given a flag ring setting $(\Lambda,\M)$, a \emph{Galois flag ring} with respect to this setting is a subalgebra $F$ of $\L$ such that 

\begin{enumerate}

\item $\Lambda \subset F$ and \ \item $LF=\L$.

    \item If moreover for any $X \in F$, $X(\Lambda) \subset \Lambda$, we have a \emph{principal Galois flag order}.

\end{enumerate}
\end{definition}

If $(\Lambda, \M, W)$ is a Galois ring setting, we define $\K:=\L^W$, $\Gamma:=\Lambda^W$, $K=\Frac \, \Gamma = L^W$; hence the extension $L/K$ is Galois and $\operatorname{Gal}(L/K)=W$.

\begin{definition}\label{main-def-3}\cite{FO}

Let $(\Lambda, \M, W)$ be a Galois ring setting. A \emph{Galois ring} with respect to this setting is an associative algebra $U$ such that, with the notation as above:

\begin{enumerate}
    \item $\Lambda$ is a Noetherian $\Gamma$-module
    \item $\Gamma \subset U$;
    \item $KU=\K$;
    \item If moreover for every $X \in U$, $X(\Gamma) \subset \Gamma$, we have a \emph{principal Galois order}.
    
\end{enumerate}

\end{definition}

If $X \in U$ is equal to $X=\sum_{\mu \in \M} l_\mu \mu$, $l_\mu \in L$, the set of $\mu$ such that $l_\mu \neq 0$ is called the support of $X$, and written $\supp \, X$.

\begin{corollary}
    With the notation of the above definition, $\Gamma$ is a Noetherian ring.
\end{corollary}
\begin{proof}
    This is an immediate consequence of the Eakin-Nagata Theorem \cite[Theorem 3.7(i)]{Matsumura}
\end{proof}

In this paper we will not discuss Galois flag rings, as they are already studied in depth in \cite{Webster} \cite{Hartwig2}, and so we focus on Galois rings.

For a non-modular Galois ring setting, we have that $\Lambda$ is a Noetherian $\Lambda^W$-module by Theorem \ref{MS}, so item (1) of Definition \ref{main-def-3} is automatically satisfied. In case $\Lambda$ is an affine $\k$-algebra, we don't need that $|W| \in \k^\times$: it follows from Noether's theorem in invariant theory that $\Lambda$ is a finite $\Lambda^W$-module, and $\Lambda^W$ is an affine, and hence Noetherian, algebra (\cite[Theorem 3.1]{Dolgachev}).

In \cite{Hartwig2} it was assumed that the Galois ring setting is always non-modular. We are going to show that many important results, analogs of those of \cite{FO} and \cite{Hartwig}, hold without this assumption. Many arguments are borrowed from these sources, so we have strived to make the proofs clearer than in the original sources.

For $\mu \in \M$, denote by $W_\mu=\{w \in W|w.\mu=\mu\}$. If $a \in L^{W_\mu}$ define $[a \mu] \in \K$ as $\sum_{w \in W/W_\mu} w(a) w.\mu$, where the sum clearly does not depend on the choice of coset representatives.

\begin{proposition}\label{prop-Hartwig-1}
    Let $\mu, \nu$ belong to $\M$.
    \begin{enumerate}
        \item $K\mu(K)=L^{W_\mu}$;
        \item $K \mu K = \{ [a \mu]|a \in L^{W_\mu} \}$, and this is a simple $K$-bimodule;
        \item $\Gamma [\mu] K= K [\mu]\Gamma=K[\mu]K$, and so $K \mathcal{V}=\mathcal{V}K$ for any sub-$\Gamma$-bimodule $\mathcal{V} \subset \mathcal{L}$.
        \item We have a decomposition of $\K$ as a finite direct sum of simple $K$-bimodules $\K=\bigoplus_{\mu \M/W} K[\mu]K$.
    \end{enumerate}
\end{proposition}
\begin{proof}
    (1) For $w \in W$, $w \in \operatorname{Gal}(L/K \mu(K))$ if and only if $w$ fixes both $K$ and $\mu(K)$, which happens if and only $w.(\mu)|_K=\mu|_k$, or equivalently, if $(w \mu w^{-1}) \mu^{-1} \in W$. By the separating assumption, this holds if and only if $w \in W_\mu$.
    
    (2) For any $ a \in L^{W_\mu}$, $K[a \mu] K= [K\mu(K) a\mu]=[L^{W_\mu} a \mu]$ by item (1).
    
    (3) Each element of $L$ is algebraic over $K$, so $\mu(\gamma)$ is algebraic over $K$ for every $\gamma \in \Gamma$. Hence $\mu(\gamma)^{-1} \in K[\mu(\gamma)]$, and hence $ K  \mu(\Gamma)=K \mu(K)$. The result now follows from (1) and (2).

    (4) Let $y \in \K$ be arbitrary, $y=\sum_{\mu \in \M} l_\mu \mu$. As $y$ is fixed by $W$, for every element $w$ of the group, $l_\mu=l_{w.\mu}$. Grouping the summands by $W$-orbits in $\M$, the result follows.
\end{proof}

\begin{corollary}\label{corol-Hartwig-1}
    If $U$ is a $\Gamma$-Galois ring in $\K$, and $\mathcal{V}$ is any $\Gamma$-subbimodule of $\L$, then $\Gamma \mathcal{V} K= K\mathcal{V} \Gamma=K\mathcal{V}K$. Also we have $UK=KU=\K$.
\end{corollary}
\begin{proof}
    Everything follows immediately from Proposition \ref{prop-Hartwig-1}(3).
\end{proof}

\begin{proposition}\label{prop-Hartwig-2}

Let $\XX \subset \K$ and let $U$ be the subring of $\K$ generated by $U$ and $\XX$. Then $U$ is a Galois $\Gamma$-ring if and only if $\bigcup_{x \in \XX} \supp x$ generates $\M$ as a monoid.    
\end{proposition}
\begin{proof}
    Necessity is clear, for if $x, y \in \L$, $\supp(xy) \subset \supp(x) \supp(y)$. For sufficiency, Proposition \ref{prop-Hartwig-1}(3)(4) implies that it suffices to show that if $x \in \XX$ and $\mu \in \supp \, x$ then $[\mu] \in KxK$. We can write $x= \sum_{\mu \in \M/W} [a_\mu \mu]$ for certain $a_\mu\in L^{W_\mu}$. By Proposition \ref{prop-Hartwig-1}(1)(2), $KxK \subset \sum_{\mu \in \M/W} K[\mu]K=\bigoplus_{\mu \in \M/W} K[\mu]K$. Since for any $\mu \in \supp \, x$, as $K[\mu]K$ is a simple $K$-bimodule, we must have $KxK \cap K[\mu]K=K[\mu]K$, which finishes the proof.
\end{proof}

\begin{definition}\cite{FO}
    Suppose that for every finite dimensional left or right $K$-vector subspace $V$ of $\K$, $U \cap V$ is a finitely generated left and right $\Gamma$-module. Then the Galois $\Gamma$-ring $U$ is called a \textit{Galois order}.
\end{definition}

The next result is of course necessary for a consistent terminology:

\begin{theorem}
    If $U$ is a principal $\Gamma$-order in $\K=\L^W$, then $U$ is a Galois order.
\end{theorem}
\begin{proof}
    $LV$ is a finite dimensional left/right vector space of $\L$. If $LV \cap U$ is a finitely generated $\Gamma$-module, as $\Gamma$ is Noetherian, so is $V \cap U$. Hence it is enough to show that for every finite dimensional $L$ vector space $V$ of $\L$, $V \cap U$ is a finitely generated $\Gamma$-module. as $L \mu = \mu L$ for $\mu \in \M$, the condition is left-right symmetric. There is no loss in generality, as $\Gamma$ is Noetherian, to assume that $V=\oplus_{i=1}^n L \mu_i, \mu_i \in \M$.

    Definite a $\Gamma$-bimodule homomorphism $\phi: U \cap V \rightarrow L^n$, $\sum_{i=1}^n l_i \mu_i \mapsto (l_1,\ldots,l_n)$, where we view $L^n$ as a bimodule by $\gamma_1 (l_1, \ldots, l_n) \gamma_2 = (\gamma_1 l_1 \mu_1(\gamma_2), \ldots, \gamma_1 l_n \mu_1(\gamma_2))$. If $(a_1, \ldots, a_n)$ belong to the image of $\phi$, as $U$ is a principal Galois order, $\sum a_i \mu_i(\gamma) \in \Gamma$, $\forall \gamma \in \Gamma$. By \cite[Lemma 2.17(ii)]{Hartwig} and the separation condition in Definition \ref{main-def-1}, there are $\gamma_1, \ldots, \gamma_n$ with, defining a $n \times n$ matrix $A=(\mu_i(\gamma_j))$ , we have that $d=\operatorname{det} A \neq 0$. Hence if we set $a=(a_1, \ldots, a_n)$, $A.a=\bar{\gamma}$, where $\bar{\gamma} \in \Gamma^n$. The adjugate matrix of $A$, $A^*$, satisfies $A^* A= d \operatorname{Id}_n$, $d \neq 0$, so $d.a=A^* \bar{\gamma}$ and so $a \in (1/d) \Lambda^n$. That is, the image of $\phi$ is contained in $(1/d) \Lambda^n$, which is a finitely generated left and right $\Gamma$ module for the bimodule structure introduced in this proof. Since $\Gamma$ is Noetherian, we finally conclude that $U \cap V$ is a finitely generated left and right $\Gamma$-module.
\end{proof}

We have the following important characterization of Galois orders:

\begin{definition}
    Let $U$ be a Galois $\Gamma$-ring. Let $M \subset U$ be a finitely generated left and right $\Gamma$-module, and $D^r(M)$ be the set of $u \in U$ such that there exists $\gamma \in \Gamma^\times$ such that $\gamma u \in M$, and $D^l(M)$  the set of $u \in U$ such that for some $\delta \in \Gamma^\times$ $u \delta \in M$
\end{definition}

\begin{proposition}\label{Galois-orders-D-sets}
Let $U$ be a Galois $\Gamma$-ring. Then $U$ is a Galois $\Gamma$-order if and only if $D^r(U)$ is a finitely generated right $\Gamma$-module and $D^l(U)$ is a finitely generated $\Gamma$-module.
\end{proposition}
\begin{proof}
    We will show only the right side case, by symmetry. The main point is to show that $D^r(M)=MK \cap U$ for right $\Gamma$-modules $M \subset U$. If $u \in D^r(M)$, for some $\gamma \in \Gamma^\times$, $u \gamma = v, v \in M$. So $u = v \gamma^{-1} \in MK$, and hence $ D^r(M) \subset MK \cap U$. If $u \in U$ belongs to $MK$, we can assume it is of the form $v\gamma^{-1}$ for some $v \in M, \gamma \in \Gamma^\times$. But then $u \gamma \in M$ and hence $u \in D^r(M)$ and so $D^r(M)=MK \cap U$. Now the result follows from the definition of Galois orders.
\end{proof}

\begin{corollary}\label{projective}
     If a Galois $\Gamma$-ring $U$ is a projective $\Gamma$-module, then $U$ is Galois order.
\end{corollary}
\begin{proof}
    If $U$ is projective, for some $\Gamma$-module $V$, $V \oplus U \simeq \bigoplus_\I \Gamma$. Let $M$ be a finitely generated $\Gamma$-submodule of $U$. Then for some finite subset $\I_0 \subset \I$, $M \subset \bigoplus_{\I_0} \Gamma$. Let $\mathbb{D}^r$ be the set of $v \in  \bigoplus_{\I_0} \Gamma$ such that there exists $\gamma \in \Gamma^\times$ with $ v \gamma\in M$; and similarly $\mathbb{D}^l$. $D^r(M) \subset \mathbb{D}^r \subset \bigoplus_{\I_0} \Gamma$. As the bigger module is a finitely generated $\Gamma$-module, so is $D^r(U)$. Similarly, $D^l(U)$ is finitely generated. Hence by the previous Proposition, $U$ is a Galois order.
\end{proof}

The following criteria is new.

\begin{corollary}\label{flat+fp}
    If $U$ is a finitely generated and flat $\Gamma$-module, $U$ is a Galois order.
\end{corollary}
\begin{proof}
    As $\Gamma$ is Noetherian, $U$ is a finitely presented module, and since it is also flat, it is projective (\cite[Exercise 2.11.8]{Rowen}).
\end{proof}

In \cite{Hartwig2} it was shown that for a non-modular principal Galois $\Gamma$-order, $\Gamma$ is a maximal (with respect to inclusion) commutative subalgebra. We show now that the result holds without the non-modularity assumption.

\begin{lemma}
    Let $U$ be a principal Galois $\Gamma$-order. Let $U_{-}:=\{ X \in U|X(1_\Gamma)=0\}$. Then $U_{-}$ is a $\Gamma$-submodule of $U$ and $U=\Gamma \oplus U_{-}$.
\end{lemma}
\begin{proof}
    If $X \in U$, write $Y$ for $X-X(1_\Gamma)$. $Y(1_\Gamma)$=0, and so each $X \in U$ decompose as $X(1_\Gamma)+(X-X(1_\Gamma)) \in \Gamma + U_-$. And $\Gamma \cap U_-=0$, since $0\neq \gamma \in \Gamma$ is not the $0$ operator on $1_\Gamma$, but $X\in U_-$ acts as 0.
    
\end{proof}

\begin{proposition}\label{Gamma-is-maximal}
    Let $U$ be a principal Galois $\Gamma$-order. Then $\Gamma$ is maximal commutative subalgebra of $U$.
\end{proposition}
\begin{proof}
    We will use the above Lemma. Let $X \in U$ and $Y=X-X(1_\Gamma)$. If $X\gamma=\gamma X$ for all $\gamma \in \Gamma$, then it is enough to show that this implies $Y=0$.

    \[ Ya= X(a)-X(1_\Gamma)a=[Xa-aX](1_\Gamma)=0.\]
\end{proof}

The following proposition will be used frequently when we discuss fixed rings of generalized Weyl algebras.

\begin{proposition}\label{main-prop} We have:
\begin{enumerate}
    \item Let $U$ be a principal Galois $\Gamma$-order, and let $G$ be a finite group of automorphisms of $U$ such that $G(\Gamma) \subset \Gamma$. If $U^G$ is a Galois ring over $\Gamma^G$, then $U^G$ is a principal Galois order.
    
    \item Let $U$ be an associative algebra and $\Gamma \subset U$ an affine commutative subalgebra. Let $G$ be a finite group of automorphisms of $U$ such that $G(\Gamma) \subset \Gamma$. If $U$ is a projective $\Gamma$-module and $\Gamma$ is a projective $\Gamma^G$-module, then $U^G$ is a projective $\Gamma^G$-module.
    \end{enumerate}
\end{proposition}
\begin{proof}
    (1): since $U^G \subset U$ and the latter is a principal Galois order, then $U^G(\Gamma) \subset \Gamma$. On the other hand, by \cite[Lemma 2.19]{Hartwig}, $U^G(K^G) \subset K^G$. Hence the claim follows. The proof of (2) is the same as \cite[Lemma 11]{FS3}.
\end{proof}

\begin{proposition}
    If $U$ is a principal Galois $\Gamma$-ring, then $\Gamma$ is a Harish-Chandra subalgebra.
\end{proposition}
\begin{proof}
First note that, as $\Gamma \subset U \subset \L$, if $\Gamma$ is a Harish-Chandra subalgebra of $\L$, then it is a Harish-Chandra subalgebra of $U$. To show that $\Gamma$ is a Harish-Chandra subalgebra of $\L$, apply Theorem \ref{theorem-quasi-central}, with $A=\L$ and $\mathsf{X}=\M$.
\end{proof}

If $U$ is not a principal Galois order, the following proposition tells us when $\Gamma$ is a Harish-Chandra subalgebra and how this notion is related with $U$ being a Galois order.

\begin{proposition}\label{Gamma-HC}
Let $U$ be a Galois $\Gamma$-ring, with $\Gamma$ an affine $\k$-algebra, and integral closure on its field of fractions $\bar{\Gamma}$. Then $\Gamma$ is a Harish-Chandra subalgebra if and only if $\mu \bar{\Gamma}=\bar{\Gamma}$ for any $\mu \in \M$. In particular if $U$ is a Galois $\Gamma$-order, then $\Gamma$ is a Harish-Chandra subalgebra.
\end{proposition}
\begin{proof}
    The same proofs of \cite[Propostion 5.1, Corollary 5.4]{FO} work in our setting.
\end{proof}

The following theorem is an improvement of \cite[Corollary 3.7]{Jauch2}.

\begin{theorem}\label{main-section-3}
    Let $U_i$ be a Galois $\Gamma_i$-ring in $(L_i*\M_i)^{W_i}$ with $\Lambda_i$ an affine algebra over an algebraically closed field $\k$, $i=1,2$. Then $U:=U_1 \otimes U_2$ is a Galois $\Gamma:=\Gamma_1 \otimes \Gamma_2$-ring in $(L*\M)^W$, $L=\Frac \, \Lambda_1 \otimes \Lambda_2$, $W=W_1 \times W_2$, $\M=\M_1 \times \M_2$. If $U_i$ is a Galois order $i=1,2$, then $U_1 \otimes U_2$ is a Galois order. If $U_i$ is a principal Galois order, so is $U_1 \otimes U_2$.
\end{theorem}
\begin{proof}
    The crucial observations are that, under these conditions, $\Lambda_1 \otimes \Lambda_2$ is a domain \footnote{Actually, $\Lambda_1 \otimes \Lambda_2$ is a domain if one of the $\Lambda_i$ is a separable $\k$-algebra and one of them, maybe the same, has $\Frac \, \Lambda_i$ a primary extension of $\k$, by \cite[Chapter V, $\oint$ 17, Corollary to Prop 1]{Bourbaki}. Our theorem holds with this much weaker hypothesis.}, and also affine and hence Noetherian. Then the first claim is \cite[Theorem 4.2]{FO}, with the same proof. For the second claim, let $V$ be a finite dimensional $K$-vector space in $(L*\M)^W$. If $LV \cap U$ is a finitely generated $\Gamma$-module, since $\Gamma$ is Noetherian, its submodule $V \cap U$ is a finitely generated $\Gamma$-module as well. We can assume without loss of generality that $LV=\bigoplus_{i=1}^s L \mu_i \nu_i$ for some $\mu_i \in \M_1$, $\nu_i \in \M_2$, and $s \in \N$.
    We have

    \[ U \into L_1*\M_1 \otimes L_2 * \M_2 \simeq^{f} (L_1 \otimes L_2)*\M \into L*\M\]
    
    Hence $LV \cap  (L_1*\M_1 \otimes L_2 * \M_2) \subset W_1 \otimes W_2  $, where $W_1=\bigoplus_{i=1}^s L_1 \mu_i$, $W_2=\bigoplus_{i=1}^s L_2 \nu_i$. Now, $LV \cap U \subset (W_1 \cap U_1) \otimes (W_2 \cap U_2)$, and as each $U_i$ is a $\Gamma_i$-Galois order, each $W_i \cap U_i$ is a finitely generated $\Gamma_i$-module. Hence $ W_1 \cap U_1 \otimes W_2 \cap U_2$ is a finitely generated $\Gamma$-module; hence so is $LV \cap U$, and finally $V \cap U$.

    Finally, if each $U_i$ is a principal Galois $\Gamma_i$-order, let $X \in U$ be arbitrary. Then we have $X=\sum Y_1^i \otimes Y_2^i$, where each $Y_j^i \in U_j$, $j=1,2$. Let $\gamma_1 \otimes \gamma_2$ be a decomposable tensor of $\Gamma_1 \otimes \Gamma_2$. As $ U \into L_1*\M_1 \otimes L_2 * \M_2$, $X(\gamma_1 \otimes \gamma_2)=\sum Y_1^i(\gamma_1) \otimes Y_2^i(\gamma_2) \in \Gamma \otimes \Gamma \onto^f \Gamma$. Hence $X(\Gamma) \subset \Gamma$, and $U$ is a principal Galois order.
\end{proof}

\section{Main theorem of Galois orders in prime characteristic}

Let $U$ an associative algebra and $\Gamma$ a commutative Harish-Chandra subalgebra, $\m$ a maximal ideal in the cofinite spectrum of $\Gamma$. In general we will work in the setting where $U$ is a Galois $\Gamma$-order.

\begin{definition}
    A Harish-Chandra module for the pair $(U, \Gamma)$ is a finitely generated $U$-module where $\Gamma$ acts locally finitely; or, equivalently, where we have a decomposition

    \[ M=\bigoplus_{\n \in \SpecMax \Gamma} M^\n, \]
    where $M^\n=\{v \in M| \n^{N_v} v=0, N_v \gg 0\}.$
\end{definition}

We denote the category of Harish-Chandra modules by $\H \C(U,\Gamma)$. Given $M$ a Harish-Chandra module, we denote by $\Supp \, M$ the set of maximal ideals $\n$ of $\Gamma$ such that $M^\n \neq 0$, and if $\dim M^\n < \infty$, we call this dimension the \textit{Harish-Chandra multiplicity} of $\n$.

Let $\Irred(\n)$ be the isomorphism set of irreducible Harish-Chandra modules $M$ with $\n \in \Supp \, M$. The main questions addressed by the Main Theoren of Galois order theory \cite[Main Theorem]{FO2} are: (a) is $\Irred(\n)$ non-empty for every $\n \in \SpecMax \, \Gamma$? (b) Is $|\Irred(\n)|$ finite? (c) If $M \in \Irred(\n)$, is $M^\n$ finite dimensional? The main result of Futorny and Ovsienko \cite{FO2} is that, under a minor technical assumption, all these questions have a positive solution to Galois orders.

We first notice that, for principal Galois orders, the answer for (a) is positive by a simple argument.

\begin{theorem}\label{lifting-principal-Galois-orders}
    Let $U$ be a principal $\Gamma$-Galois order.
    \begin{enumerate}
        \item Every maximal ideal $\m$ of $\Gamma$ is contained in a maximal right ideal $M$ of $U$.
        \item If $\m \in \cfs \, \Gamma$, then there is an irreducible Harish-Chandra $U^\circ$-module $\mathsf{M}$ with $\m$ in its support.
    \end{enumerate}
\end{theorem}
\begin{proof}
    (1) We show that $\m U$ is a proper right ideal of $U$, and then, by Zorn's Lemma, we will have our maximal right ideal $M$. We have a map $\phi: U \rightarrow \Gamma/\m$ which is the composition $\theta \circ \psi$ of $\psi:U \onto \Gamma$, where $X \in U$ goes to $X(1_\Gamma)$, and $\theta: \Gamma \rightarrow \Gamma/\m$ is the canonical projection. The map $\phi$ is surjective and has $\m U$ in its kernel. So we have an epimorphism $U/\m U \rightarrow \Gamma/\m$, and so $\m U \neq U$.
    (2) $1+M \in \mathsf{M}$, where $\mathsf{M}=U/M$, lies in $\mathsf{M}^\m$. Since $\mathsf{M}$ is irreducible and generated by a generalized weight vector ($1+M$), \cite[Proposition 14]{DFO} implies that $\mathsf{M}$ is actually a Harish-Chandra module.
\end{proof}

We will now reprove the Main Theorem of Galois order theory in prime characteristic. We will follow closely the presentation in \cite{FO}, \cite{FO2}. We also strive for a clearer exposition of some results.

\textbf{\textit{From now on}} Assume $\Lambda$ is affine, and $\k$ algebraically closed, and $\M$ a group. Assume also that $\Gamma$ integrally closed, until the proof of Theorem \ref{Main-Theorem}.

Let $S$ be a subset of $\M$ that is $G$-invariant, and $\Oo_i,i=1,\ldots,k$ its orbits. We introduce $U(S)=\{u \in U| \supp \, u \subset S\}$, a $\Gamma$-subbimodule of $U$.

\begin{lemma}\label{US}
  We have that $D^r(U(S))=U(S)$. The same for $D^l$.
\end{lemma}

In the next Proposition assume $\Gamma$ is a Harish-Chandra subalgebra. In this setting, define $f_S^r$ (and similarly $f_S^l$), an element of $\Gamma \otimes \Gamma$, by

\[ f_S^r= \prod_{s \in S} [f \otimes 1 - 1 \otimes s^{-1}(f)].\]

Note that the coefficients of $f_S^r$ must belong to $K$. The next proposition deals with $f_S^r$, but we have an obvious symmetric result for $f_S^l$.

\begin{proposition}\label{horrible-proposition}
\begin{enumerate}
    \item Assume $u \in U$. Then $u \in U(S)$ if and only if $f_S^r u=0$ for any $f \in \Gamma$.
    \item If $u \in U$ and $S^c= \supp \, u\setminus S$, then $f_{S^c}^ru\in U(S)$ for every $f \in \Gamma$.
    \item If $f_S^r=\sum f_i \otimes g_i$ ($f_i, g_i \in \Gamma)$, and $[a \mu] \in \K$, then $f_S^r[a \mu]=[\sum(f_i \mu(g_i)a)\mu]=\prod_{s \in S}[(f-s^{-1}\mu(f))a \mu]$
    \item Let $\Oo$ be a $G$-orbit of a certain $\mu$ in $\M$ and let $T \subset \M$ be a $G$-invariant subset. The $\Gamma$-bimodule homomorphism $P_\Oo^T(f)$ that sends $U(T)$ to $U(\Oo)$, $u \mapsto f_{T\setminus \Oo}^r u$, is either $0$ or has kernel $U(T\setminus \Oo)$.
    \item The map $P_S(f)=(P_{\Oo_1}^S(f), \ldots, P_{\Oo_k}^S(f))$ is a $\Gamma$-bimodule monormorphism from $U(S)$ into $\bigoplus_{i=1}^k U(\Oo_i)$.
\end{enumerate}
\end{proposition}
\begin{proof}
    The proof of \cite[Lemma 5.2]{FO} works without modification here.
\end{proof}

\begin{theorem}\label{less-horrible-theorem}
Let $U$ be a Galois $\Gamma$-ring, where $\Gamma$ is a Harish-Chandra subalgebra. Then the following are equivalent: (a) $U$ is a Galois order; (b) $U(S)$ is a finitely generated left and right $\Gamma$-module for each $G$-invariant subset $S \subset \M$; (c) for each orbit $\Oo_\mu$ of the $G$-action on $\M$, $U(\Oo_\mu)$ is a finitely generated left and right $\Gamma$-bimodule.   
\end{theorem}
\begin{proof}
    (a) implies (b): Let $u_1, \ldots, u_s$ be a $K$-basis of $U(S)K$. $D^r(\sum u_i \Gamma)= (\sum u_i \Gamma)K \cap U$ by the proof of the Proposition \ref{Galois-orders-D-sets}. By Corollary \ref{corol-Hartwig-1} this is the same as $U(S)K \cap U$, which is a finitely generated right $\Gamma$-module as $U$ is a Galois order. But, again by the proof of the Proposition \ref{Galois-orders-D-sets},  $U(S)K \cap U=D^r(U(S))$, which is just $U(S)$ by Lemma \ref{US}.

    (b) implies (c) is trivial.

    To show that (c) implies (a) we will use Proposition \ref{Galois-orders-D-sets}. Assume $M \subset U$ is a finitely generated right $\Gamma$-submodule, and let $S=\bigcup_{u \in M} \supp \, u$. $D^r(M) \subset D^r(U(S))=U(S)$, by Lemma \ref{US}. As $\Gamma$ is Noetherian, it remains to show that $U(S)$ is a finitely generated right $\Gamma$-module. But by the last item of the previous Proposition, $U(S)$ is a $\Gamma$-submodule of $\bigoplus U(\Oo_i)$ for certain orbits $\Oo$. Again using the fact that $\Gamma$ is Noetherian, we are done.
\end{proof}

Let $\Mm$ be a lift of $\m \in \SpecMax \Gamma$ to $\SpecMax \Lambda$. Let $\M_\m$  be the $\m$-stabilizer of $\Mm$ and let and assume it is finite. The finiteness of $\M_\m$ does not depend on the choice of $\Mm$ because it is unique up to $W$-conjugation. Similarly we define $W_\m$ as the $W$-stabilizer of $\Mm$.

By Hilbert's Nullstellensatz, for each $\m \in \SpecMax \, \Gamma$, $\Gamma/\m \simeq \k$. If $f \in \Gamma$, $f(\m)$, the evaluation at $\m$, is the image of $f$ in the above canonical isomorphism. If $\Mm$ is a lift from $\m$ to $\Lambda$, we set $f(\Mm):=f(\m)$. It is clearly well-defined.

Let $\m, \n \in \SpecMax \Gamma$ and $\Mm, \Nn$ lifts of them to the maximal spectrum of $\Lambda$. We define

\[ S(\m,\n)=\{\mu \in \M|(\exists w_1,w_2 \in W)(w_2 \Nn=\mu w_1 \Mm) \}.\]

This set clearly does not depend on the choice of $\Mm, \Nn$.

\begin{lemma}\label{G-T-1}
Let $\m \in \SpecMax \, \Gamma$, and assume $\M_\m$ is finite. Then $|S(\m,\m)| \leq \frac{|W|^2}{|W_\m|^2}|\M_\m|$.
    
\end{lemma}
\begin{proof}
Let $\Mm_1, \Mm_2$ be two liftings of $\m$ to maximal ideals of $\Lambda$. There are at most $
|\M_\m|$ elements $\mu \in \M$ with $\mu \Mm_1=\Mm_2$, and at most $\frac{|W|}{|W_\m|}$ elements $w \in W$ with $w \Mm_1 = \Mm_2$. The Lemma follows.
\end{proof}

\begin{lemma}\label{G-T-2}
Let $U$ be a Galois $\Gamma$-ring with $\Gamma$ a Harish-Chandra subalgebra, and pick $\m \in \SpecMax \, \Gamma$ such that $\M_\m$ is finite. Set $T=S(\m,\m)$. If $U=U\m$, then for every $k \geq 0$, there are $\gamma_k \in \Gamma \setminus \m$, $u_j \in U$, $\chi_j \in \m^k$, $j=1, \ldots, N$, such that $\gamma_k =\sum_{j} u_j \chi_j$, and $\supp \, u_j \subset T$
\end{lemma}
\begin{proof}
The condition $U=U\m$ just says that $1 \in U\m$, and hence

\[ (\dagger) \, 1=\sum_i w_i \zeta_i, w_i \in U, \zeta_i \in \m.\]

We will use $(\dagger)$ as the base of a induction argument to show the statement of the lemma without assuming $\supp \, u_j \subset T$ yet. So assume $k > 1$ and assume that there  $u_j \in U$, $\chi_j \in \m^k$, $j=1, \ldots, N$, such that $1 =\sum_{j} u_j \chi_j$. Then

\[ (\ddagger) \, 1=\sum_{j} u_j \chi_j =\sum_{j} u_j (\sum_i  w_i \zeta_i)\chi_j = \sum_j \sum_i u_j w_i (\chi_j \zeta_i), \]

and as $\chi_j \zeta_i \in \m^{k+1}$ we are done for the first half of the proof.

Now, set $Z=  \bigcup \supp \, u_j \setminus T$. As $T \cap Z = \emptyset$, for any lifting of $\m$ to $\Mm$ in $\SpecMax \, \Lambda$, we have that $z(\Mm) \neq \Mm$, for all $z \in Z$. Then there exists $f \in \Gamma$ with $f(\Mm) \neq f(z(\Mm))$ for all $z \in Z$. By Proposition \ref{horrible-proposition}(3), introducing $f_Z^r$, we can assume without loss of generality that $\prod_{z \in Z}[f(\Mm)-f(t(\Mm))]=1$. In particular $f_Z^r \in 1 + \Gamma \otimes \m + \m \otimes \Gamma$. Hence $\gamma_k=f_Z^r(1) \in 1+\m$. Replacing $u_j$ by $f_Z^r(u_j)$, we have by Proposition \ref{horrible-proposition}(2) that $u_j \in U(T)$. If $u \in U$, $\gamma \in \Gamma$, $f_Z^r(u \gamma)= f_Z^r(u) \gamma$, since $\Gamma$ is commutative and  $f_Z^r \in 1 + \Gamma \otimes \m + \m \otimes \Gamma$. Hence, applying $f_Z^r$ to $(\ddagger)$ gives us the desired result.
\end{proof}

The next theorem is a relevant improvement of Theorem \ref{lifting-principal-Galois-orders}

\begin{theorem}\label{lifting-Galois-orders}
    Let $U$ be a Galois $\Gamma$-order. Then for every maximal ideal $\m$ of $\Gamma$ with $\M_\m$ finite, there is an $U$-module $\mathsf{M} \in \Irred(n)$. Moreover, for each $v \in \mathsf{M}$, $\m v=0$.
\end{theorem}
\begin{proof}
    By Proposition \ref{Gamma-HC}, $\Gamma$ is a Harish-Chandra subalgebra. By Lemma \ref{G-T-1}, $T=S(\m,\m)$ is finite. So by Theorem \ref{less-horrible-theorem}(2), $U(S)$ is a finitely generated left and right $\Gamma$-module. By the Artin-Rees Lemma (cf. \cite[Theorem 8.5]{Matsumura}), there exists an $n_0 \in \N$ such that for every $n \geq n_0$, $U(S)\m^n \cap \Gamma=(U(S)\m^{n_0}\cap \Gamma)\m^{n-n_0}$.

    Lets show that $U= U \m$ is impossible. Otherwise, by Lemma \ref{G-T-2}, for every $n>n_0$ there exist $\gamma_n \notin \m$ that belongs to  $U(S)\m^n \cap \Gamma$. But we can't have $\gamma_n$ belonging to $(U(S)\m^{n_0}\cap \Gamma)\m^{n-n_0}$, which is the same set. Hence, $U \neq U\m$.

    In this way, the module $U/U \m$ has a simple quotient $\mathsf{M} \in \Irred(\m)$, just like in the proof of Theorem \ref{lifting-principal-Galois-orders}. The module is generated by $1+M$, in the notation of the proof of that theorem. Hence $\m \mathsf{M}=0$.
\end{proof}

The next Lemma has the same proof as \cite[Lemma 4.6]{FO}:

\begin{lemma}\label{G-T-3}
    Let $\m, \n$ be maximal ideals of $\Gamma$ and $T=S(\m,\n)$. For every $m,n \geq 0$, $U=U(T)+\n^m U+U \m^n$. Moreover, for every $u \in U$ and $k \geq 0$ there is $u_k \in U(T)$ such that $u \in u_k + \m^k u \Gamma + \Gamma u \m^k$.
\end{lemma}

We now prepare for proof of the Main Theorem. We recall some facts from \cite{DFO}. Let $U$ be an associative algebra, $\Gamma$ a commutative Harish-Chandra subalgebra.

\begin{definition}
    Let $\m \in \SpecMax \, \Gamma$. Let $\widehat{\Gamma}_\m$ be the $\m$-adique completion of $\Gamma$ at $\m$ and let $\A(\m,\m)$ be the $\widehat{\Gamma}_\m$-bimodule $\varprojlim_{m,n} U/(\m^m U + U \m^n)$
\end{definition}

The following is \cite[Corollary 19]{DFO}

\begin{theorem}\label{main-DFO}
Let $U$ be an associative algebra and $\Gamma$ a commutative Harish-Chandra subalgebra. Let $\m \in \cfs \, \Gamma$. If $\A(\m,\m)$ is a finitely generated left and right $\widehat{\Gamma}_\m$ module, then $\Irred(\m)$ is finite and for every $\mathsf{M} \in \Irred(\m)$, $\dim \mathsf{M}^\m < \infty$.
\end{theorem}

We are ready to prove the main result of Galois orders theory in prime characteristic.

\begin{theorem}\label{Main-Theorem}
    Let $U$ be a Galois $\Gamma$-order, and $\m$ a maximal ideal of $\Gamma$ such that $\M_\m$ is finite. Then $\Irred(\m)$ is finite and for every $\mathsf{M} \in \Irred(\m)$, $\dim \mathsf{M}^\m < \infty$.
\end{theorem}
\begin{proof}
    Our strategy is to use the previous Theorem. Begin assuming $\Gamma$ integrally closed. By Lemma \ref{G-T-3}, $\A(\m,\m) \subset \varprojlim_{m,n} U(S)/ (\m^m U + U \m^n) \cap U(S)$. By Theorem \ref{less-horrible-theorem}, $U(S)$ is a finitely generated left and right $\Gamma$-module, and hence so is each $ U(S)/ (\m^m U + U \m^n)$, and the generators of $U(S)$ as left or right module generate  $U(S)/ (\m^m U + U \m^n)$ as well. As $\widehat{\Gamma}_\m$ is Noetherian (\cite[Theorem 10.26]{AM}) and $\varprojlim_{m,n} U(S)/ (\m^m U + U \m^n) \cap U(S)$ a finitely generated left and right $\widehat{\Gamma}_\m$-module (\cite[Theorem 8.7]{Matsumura}), $\A(\m,\m)$ is a finitely generated left and right $\widehat{\Gamma}_\m$-module. Hence we are done by Theorem \ref{main-DFO}. If $\Gamma$ is not integrally closed, its integral closure in $K$, $\bar{\Gamma}$ is a finite $\Gamma$-module, and hence each maximal ideal of $\Gamma$ lift to a finite number of maximal ideals of $\bar{\Gamma}$. Hence we can assume $\Gamma$ integrally closed without loss of generality.
\end{proof}

\section{Examples and counter-examples}

This very small section is dedicated to two counter-examples. Essentially all known Galois rings are actually principal Galois orders (cf. \cite{Hartwig}, \cite{Webster}, \cite{Schwarz2}).

The first natural examples of Galois rings that are not Galois are not Galois orders are the alternating analogues of $U(\mathfrak{gl}_n)$, the algebras $\AAA(\mathfrak{gl}_n)$, for $n>2$, introduced in \cite{Jauch}.

Our first counter-example in this section is a Galois ring that is not a Galois order that is elementary in comparison with the example in \cite{Jauch}. Let $U=\k[x^{\pm 1}]$, $\Gamma=\k[x]$, with $\M$ and $W$ trivial, and $K=\K=\L=\k(x)$. $U$ is clearly a Galois $\Gamma$-ring in $K$. It is not a Galois order, for if  $V$ is any $K$-vector space in $K$, $K \cap U= U$, which is not a finitely generated left or right $\Gamma$-module.

Now we will show a simple family of examples of Galois orders which are not principal.

\begin{proposition}\label{finite-module}(\cite[Proposition 2.1]{FO})
If $A$ is an affine algebra and an integral domain, and $\bar{A}$ denotes its integral closure in $\Frac \, A$. Then $\bar{A}$ is a finite $A$-module.
\end{proposition}
\begin{proof}
    It is immediate from Noether's normalization lemma.
\end{proof}

\begin{lemma}
    Let $B \subset A$ be two rings, with $A$ a commutative domain and a finite $B$-module. Then $\Frac \, A = A (B^\times)^{-1}$, where $B^\times=B \setminus \{ 0 \}$.
\end{lemma}
\begin{proof}
    It is clear that $A (B^\times)^{-1}$ is a finite dimensional $\Frac \, B$-vector space. To show that it equals $\Frac \, A$, it is enough to show that any $0 \neq a \in A$ is already invertible in $A (B^\times)^{-1}$. Left multiplication by $A$ is a $\Frac \, B$ linear endomorphism of  $A (B^\times)^{-1}$, which is injective since $A$ is a domain. Since the dimension is finite, the map is surjective as well. Hence there is $a'b^{-1}$, with $a\neq 0 \in A, b \neq 0 \in B$ an inverse for $a$.
\end{proof}

In the next family of examples, again $\M$ and $W$ are trivial.

\begin{theorem}\label{not-principal-Galois-order}
    Let $A$ be an affine commutative domain which is not integrally closed, $\bar{A}$ its integral closure, and $K=\Frac \, A= \Frac \, \bar{A}$. Then $\bar{A}$ is a Galois $A$-order which is not principal.
\end{theorem}
\begin{proof}
    By the previous Lemma, we have that $\bar{A}$ is a Galois $A$-order in $K$. $K$ is the only $K$-vector space in $\K=\L=K$. $\K \cap \bar{A}=\bar{A}$, which by Proposition \ref{finite-module} is a finitely generated $A$-module.
\end{proof}

We will end this section with two examples that show how varied can be the skew monoid rings $\L$ and their fixed rings $\K$. Both examples also hold for any algebraically closed field of any characteristic.

Our first example comes from elliptic curves, and we will use \cite{ellipt} as our main source.

\begin{lemma}
    Let $E$ be an elliptic curve. Then it group of automorphisms (as an abelian group and as an algebraic variety at the same time) is finite. If we denote this group by $G$, then $\k(E)^G \simeq \k(\mathbb{P}^1)$.
\end{lemma}
\begin{proof}
The first claim is \cite[Theorem III.10.1]{ellipt} and the second one follows from a particular form of the Hurwitz formula (e.g., \cite[Exercise 3.13]{ellipt}).
\end{proof}

\begin{proposition}\label{elliptic}
    Let $\k$ be any algebraically closed field of finite characteristic, $E$ and elliptic curve, and  Let $\M$ be the subgroup of $\Autk_\k \, \k(E)$ consisting of all $m$-torsion points, $m \in \N$. Let $G:= \Autk \, E$; it acts on $\M$ by conjugation ($gPg^{-1}=g(P)$, $P \in \M$). Choose a set of generators $\XX$ of $\M$. Then if $U$ is subring of $(\k(E)*\M)^G$ generated by $\k[\mathbb{A}^1]$ and $\XX$, then $U$ is a Galois $\k[\mathbb{A}^1]$-ring in $(\k(E)*\M)^G$.
\end{proposition}
\begin{proof}
    The result follows from the previous Lemma and Proposition \ref{prop-Hartwig-2}.
\end{proof}

For our second example, let $G$ be a connected linear algebraic group. We form $\L=\k(G)*G$, identifying $x \in \k G$ with $\lambda_x$, $\lambda_x f(y)=f(x^{-1}y)$. Let $W$ be a finite group of automorphisms of $G$, as a variety and as a group. Note that if $f \in \k(G)$, $x \in G$, $w \in W$, then $w \lambda_x w^{-1} f(y)=\lambda_{w(x)} f(y)$, so $W$ indeed acts by conjugation normalizing $G$. The next definition is a generalization of the notion of rational Galois order from \cite[Definition 4.3]{Hartwig}.

\begin{definition}\label{linear-algebraic-Galois-ring}
    A linear algebraic Galois ring is a Galois $\k[G]^W$-ring that embeds in some $(\k(G)*G)^W$.
\end{definition}

\begin{example}\cite{FH}
In the Galois ring realization of $U_q(\mathfrak{gl}_n)$, $q$ not a root of unity, the skew group ring is $(\C(x_1,x_2,\ldots,x_{n(n-1)/2}; z_1, \ldots, z_n)*\Z^{n(n-1)/2})^{D_n}$, where $D_n$ is the dihedral group $\{ (a_1\sigma_1, \ldots, a_n\sigma_n) \in \Z_2^N \rtimes S_n| \sum a_i=0 \}$. $\Z^{n(n-1)/2})$ has a basis $\varepsilon_{ij}$. $i \in [1,n-1]$, $j \in [1, i]$, such that it fixes the $z$ and act by $\varepsilon_{ij}(x_{kl})=q^{- \delta_{ik} \delta_{jl}} x_{kl}$. It is a linear algebraic Galois $\C[\mathbb{T}^{n^2}]^{D_n}$-ring where $G=\mathbb{T}^{n^2}=\C^{\times n^2}$.
\end{example}

Similarly, all the Galois orders from \cite{FS2} are generalized rational Galois rings with $G=\C^{\times n}$.

\begin{remark}
    If $G$ is a connected linear algebraic group, then it is always rational.
\end{remark}
\begin{proof}
Let $G$ be a such a group with unipotent radical $R_u$. Let $B$ be a borel subgroup and $U$ its unipotent radical (so that $R_u < U$). There is a dense $B$-orbit $V$ in $G/B$ which is a rational variety. Since $U$ and $R_u$ are split unipotent, by \cite[14.2.5]{Springer} there exists a section $s:U/R_u \rightarrow U$ to the projection $\pi: U \rightarrow U/R_u$. Using $s$ one can find a local section to the projection $\Pi :G \rightarrow G/B$. This shows that, in the Zarisk topology, $f$ is a locally trivial $B$-bundle, and in particular, $f^{-1}(V)$ is an open subvariety of $G$ isomorphic to the rational variety $V \times B$.
\end{proof}

\section{Generalized Weyl algebras and their infinite rank analogues}

There are many ways in which the idea of the Weyl algebra can be extended. Let us recall the definition of a generalized Weyl algebra (henceforth denoted GWA), due to V. Bavula \cite{Bavula}. Our rings in this section are $\k$-algebras for an arbitrary base field $\k.$

\begin{definition}\label{finite-GWA}
    Let $D$ be a ring, and $\sigma=(\sigma_1, \ldots, \sigma_n)$ a $n$-uple of commuting automorphisims: $\sigma_i \sigma_j = \sigma_j \sigma_i$, $i,j=1,\ldots,n$. Let $a=(a_1,\ldots,a_n)$ be a $n$-uple of non zero elements belonging to the center of $D$, such that $\sigma_i(a_j)=a_j, j \neq i$. The \emph{generalized Weyl algebra} $D(a, \sigma)$ of degree $n$ and base ring $D$ is generated over $D$ by $X_i^+, X_i^-$, $i=1,\ldots, n$ and relations
    \begin{subequations}\label{eq:GWA-relations}
    \begin{gather}
    X_i^+(d)= \sigma_i(d) X_i^+, \qquad  X_i^- d= \sigma_i^{-1}(d) X_i^-,\quad \forall d \in D,\\
    [X_i^+, X_j^+]=[X_i^-,X_j^-]=[X_i^+, X_j^-]=0, \;\forall i\neq j,\\
    X_i^- X_i^+ = a_i, \qquad X_i^+  X_i^- = \sigma_i(a_i).
    \end{gather}
    \end{subequations}
\end{definition}

\begin{definition}
    Let $\alpha$ be an ordinal number. By $\Z^{\oplus \alpha}$ we mean the subgroup of $\prod_{\beta < \alpha} \Z_\beta$ of sequences of finite support, where each $\Z_\beta=\Z$.
\end{definition}

\begin{proposition}\label{GWA-basic}
    If $D$ is a domain, $D(a,\sigma)$ is a domain. If $D$ is Noetherian, $D(a,\sigma)$ is Noetherian.
\end{proposition}
\begin{proof}
    \cite[Proposition 1.3]{Bavula}.
\end{proof}

Now we introduce the notion of an infinite rank generalized Weyl algebra, particular cases considered in \cite{BBF} and \cite{FGM}). It was discussed \textit{en passant} in the preprint \cite{Schwarz2}. Here we develop the basics of their theory. One main difference is that our theory uses ordinal numbers as degrees of the infinite rank GWAs, and not cardinals as in \cite{Schwarz2}.

\begin{definition}\label{infinite-GWA}
    Let $D$ be a ring, $\alpha$ and ordinal number and $\sigma=\langle \sigma_\beta\rangle_{\beta < \alpha}$ a sequence of length $\alpha$ of commuting automorphisims: $\sigma_i \sigma_j = \sigma_j \sigma_i$, $i,j<\alpha$. Let $a=\langle a_\beta \rangle_{\beta < \alpha}$ be a sequence of length $\alpha$ of non-zero elements belonging to the center of $D$, such that $\sigma_i(a_j)=a_j, j \neq i$. The \emph{generalized Weyl algebra} $D(a, \sigma)$ of degree $\alpha$ and base ring $D$ is generated over $D$ by $X_\beta^+, X_\beta^-$, $\beta<\alpha$ and relations
    \begin{subequations}\label{eq:GWA-relations2}
    \begin{gather}
    X_\beta^+(d)= \sigma_\beta(d) X_i^+, \qquad  X_\beta^- d= \sigma_\beta^{-1}(d) X_\beta^-,\quad \forall d \in D,\\
    [X_\beta^+, X_\gamma^+]=[X_\beta^-,X_\gamma^-]=[X_\beta^+, X_\gamma^-]=0, \;\forall \beta \neq \gamma,\\
    X_\beta^- X_\beta^+ = a_\beta, \qquad X_\beta^+  X_\beta^- = \sigma_\beta(a_\beta).
    \end{gather}
    \end{subequations}
\end{definition}

\textbf{Convention}: We will call the generalized Weyl algebras and their infinite rank generalizations as just GWAs. If for a degree $\alpha$ GWA, all elements in $a=\langle a_\beta \rangle_{\beta < \alpha}$ are regular in the base ring $D$, we call the GWA \textit{regular}.

\begin{definition}
    Let $D(a,\sigma)$ be a GWA of degree $\alpha$ and base ring $D$. Let $\Sigma$ be the the subgroup of $\Autk \, D$ generated by $\langle \sigma_\beta \rangle_{\beta < \alpha}$. There is a natural homomorphism $\xi: \Za \rightarrow \Sigma$. which is clearly surjective. Denote the image of $z \in \Za$ by $\sigma_z$. If $\xi$ is an isomorphism, we say that the GWA is of \textit{sujective type}.
\end{definition}

\begin{proposition}\label{quasi-central-GWA}
If $D(a,\sigma)$ is a GWA of finite or infinite degree with base ring $D$, then $D$ is a quasi-central subalgebra. Hence, if $D$ is commutative, it is a Harish-Chandra subalgebra.    
\end{proposition}
\begin{proof}
    An immediate consequence of Proposition \ref{proposition-quasi-central}.
\end{proof}

Let $X_i^n$ denote $X_i^{+ n}$ if $n \geq 0$; and $X_i^{- (-n)}$ otherwise.. If $z \in \Z^{\oplus \alpha} $ is equal to $m_{\beta_1}+m_{\beta_2}+\cdots+m_{\beta_s}$, where each $m_{\beta_i} \in \Z$, define $X^z:= X_{\beta_1}^{m_{\beta_1}} X_{\beta_2}^{m_{\beta_2}} \cdots X_{\beta_s}^{m_{\beta_s}}$. Then we have that $D(a,\sigma)$ is clearly a $\Z^{\oplus \alpha}$-graded algebra; for $z \in \Z^{\oplus \alpha}$ the homogeneous component of degree $z$ is $D X^z$; also $D(a,\sigma)$ is a free left and right $D$-module with basis $X^z, z \in \Z^{\oplus \alpha}$.

\begin{definition}
Let $\I$ be a directed set. For each $i \in \I$, let $D_i(a^i,\sigma^i)$ be a GWA of finite degree $n_i$ and base ring $D_i$.  Suppose that for $i,j \in \I, j \geq i$ we have connecting morphisms $f_{ij} : D_i(a^i,\sigma^i) \rightarrow D_j(a^j, \sigma^j)$ with $f_{ij}(D_i) \subset D_j$, $n_j > n_i$, and for $1 \leq \ell \leq n_i$, $\sigma_\ell^j|_{D_i}=\sigma_\ell^i$, $a_\ell^i \in Z(D_j)$. This data is called a \textit{direct system of GWA's}.
\end{definition}

\begin{theorem}\label{GWA-as-limits}
    (a) Let $D(a,\sigma)$ be a generalized Weyl algebra of infinite degree $\alpha$ and base ring $D$. Let $\I$ be the family of finite subsets of $\alpha$, ordered by inclusion. For each $F \in \I$, with $|F|=n$, it has a total ordering inherited from that of $\alpha$: $\langle f_1, f_2, \ldots, f_n \rangle$. Consider the degree $n$ GWA with base ring $D$ given by $D_F(a^F,\sigma^F):=D(\langle a_{f_1}, \ldots a_{f_n} \rangle, \langle \sigma_{f_1}, \ldots, \sigma_{f_n} \rangle)$. For $F \subsetneq F'$ we have a natural morphism $f_{F,F'}: D_F(a,\sigma) \into D_{F'}(a,\sigma)$ that gives us a direct system of GWA's. Then $D(a,\sigma)=\varinjlim_{F \in \I} D_F(a^F, \sigma^F)$.
    
    Reciprocally,
    
    (b) let $\I=\omega$ be an infinite directed set and for each $i$ let $D_i(a^i,\sigma^i)$ be a finite rank GWA with connecting morphisms $f_{ij}$ for $j \geq i \in \I$ giving the data of a direct system of GWA`s. Then $\varinjlim_{i \in \omega} D_i(a^i,\sigma^i)$ is a degree $\omega$ GWA with base ring $D_\omega= \varinjlim_{i \in \omega} D_i$.
\end{theorem}
\begin{proof}
    (a)
    For each $F \in \I$, we have an inclusion $\iota_F:D_F(a^F,\sigma^F) \into D(a,\sigma)$. These inclusions satisfy the condition $\iota_{F}=\iota_{F'} f_{F, F'}$, and hence by the universal property of the inductive limit we have a morphism $\Theta: \varinjlim_{F \in \I} D_F(a^F, \sigma^F) \rightarrow D(a,\sigma)$. This map is clearly surjective, as its image contains $D$ and each $X_\beta^\pm$, $\beta<\alpha$. It is also injective: if $a \in \varinjlim_{F \in \I} D_F(a^F, \sigma^F) $, in the notaion of the proof of Proposition \ref{prop-inductive-limit}, we have that for some $F \in \I$, $a=\phi_F(a_F)$ for some $a_F \in D_F(a^F, \sigma^F)$. So $\Theta(a)= \Theta \phi_F(a_F)=\iota_F(a_F)$. Since $\iota_F$ is a monomorphism, $\Theta(a)=0$ would imply $a_F=0$ and hence $a=0$. So we have our isomorphism.

    (b)
    $D_\omega=  \bigcup_{i \in \omega} D_i$, and set $a=\langle a_i \rangle_{i < \omega}$. Fix an $\ell \geq 1$ and consider all $\sigma^i_\ell$. If we have $\sigma^j_\ell$ and  $\sigma^k_\ell$ with $k>j$, then $D_j \subset D_k$ and the restriction of of $\sigma^k_\ell$ to $D_j$ is just $\sigma^j_\ell$. With this we define an automorphism of $D$, $\sigma^*_\ell$, as follows: pick any $i$ with $n_i \geq \ell$. If $j \leq i$, $D_j \subset D_i$, and we define $\sigma^*_\ell(d)$ for $d \in D_j$ as $\sigma^i_\ell(d)$. If $j>i$, define for $d \in D_j$ $\sigma^*_\ell(d)=\sigma^j_\ell(d)$. By the compatibility condition of a directed system of GWAs, $\sigma^*_\ell$ gives a well-defined automorphism of $D_\omega$. Now setting $\sigma=\langle \sigma^*_\ell \rangle_{\ell \in \omega}$, we finally have $\varinjlim_{i \in \omega} D_i(a^i,\sigma^i)=D_\omega(a,\sigma)$ by the universal property of inductive systems.
\end{proof}


\begin{theorem}\label{GWA-simple}
Let $D(a,\sigma)$ be a degree $\alpha$ GWA, where $\alpha$ is an infinite ordinal. Assume that $D$ does not have any $\sigma$-stable ideal, that the subgroup generated by the $\sigma_i$ in the factor group $\Autk \, D/\operatorname{Inn} D$ is isomorphic to $\Za$, and that for each $i \in \alpha$, and $m \geq 1$, $D a_i + D \sigma_i^m(a_i)=1$. Then $D(a,\sigma)$ is a simple ring.
\end{theorem}
\begin{proof}
    In the notation of Theorem \ref{GWA-as-limits}(a), each $D_F(a^F,\sigma^F)$ is a simple ring, by the simplicity criteria of finit rank GWAs, \cite[Theorem 4.5]{BF}. As $D(a,\sigma)$ is an injective direct limit of these $D_F(a^F,\sigma^F)$ (Theorem \ref{GWA-as-limits}(a)), by Proposition \ref{prop-inductive-limit}(2), $D(a,\sigma)$ is a simple ring.
\end{proof}

\begin{example}\label{simple-GWA-examples}\cite[Corollary 4.8]{BF}
    The most important examples of simple GWAs are the following ($\operatorname{char} \k=0$).

    (a) Call a rank 1 GWA $A$ of classical type if $A=\k[h](a,\sigma)$, $\sigma(h)=h-1$. If there is no irreducible polynomial $p[h]$ such that both $p, \sigma^i p$ are multiples of $a$, for any $i \geq 0$, the GWA of classical type is simple.

    (b) Call a rank 1 GWA $A$ of quantum type if $A=\k[h^{\pm 1}](a,\sigma)$, $\sigma(h) = \lambda h$, $\lambda \neq 0$, $\lambda \neq 1$, $\lambda$ is not a root of unity. If there is no irreducible polynomial $p[h]$ with $p, \sigma^i(p)$ both being multiples of $a$, $i \geq 0$, the GWA of quantum type is simple.

    (c) If a generalized Weyl algebra $A$ is the tensor product $\bigotimes_{i=1}^n A_i$, where each $A_i$ is a simple GWA's of either classical or quantum type, then $A$ is also simple.    
\end{example}

\begin{corollary}\label{GWA-Ore}
    Let $D(a,\sigma)$ be a GWA of infinite degree $\alpha$. It is not Noetherian. If $D$ is a domain, then $D(a,\sigma)$ a domain. If $D$ is a Noetherian domain, $D(a, \sigma)$ is an Ore domain.
\end{corollary}
\begin{proof}
Let $\omega=\N$ be a initial segment of $\alpha$. Then the chain of ideals $(X_1^+) \subset (X_1^+ + X_2^+) \subset \ldots$ is clearly ascending and does not stabilizes. By Theorem \ref{GWA-as-limits}a), $D(a,\sigma)$ is an injective direct limit of GWAs of finite rank. As $D$ is a Noetherian domain, those are Noetherian domains, by Proposition \ref{GWA-basic}, and so $D(a,\sigma)$ is a domain by Proposition \ref{prop-inductive-limit}(1). If $D$ is a Noetherian domain, $D(a,\sigma)$ is an injective direct limit of GWAs of finite rank which are Noetherian domains (cf. Proposition \ref{GWA-basic}), and hence Ore domains, and hence $D(a,\sigma)$ is an Ore domain by Proposition \ref{prop-inductive-limit}(3).
\end{proof}

\begin{theorem}\label{GWA-localization}
    Let $\A=D(a,\sigma)$ be a degree $\alpha$ generalized Weyl algebra with base ring $D$. Let $S$ be a multiplicatively closed subset of $Z(D)$ consisting of regular elements of $D$ and such that is $\sigma$-stable: that is, $\sigma^{j}_\beta(S) \subset S$, $\forall \beta < \alpha, j \in \Z$. Let $\AAA=D_S(a,\sigma)$, where each $\sigma_\beta$ is extended as an automorphism of $D_S$ in a natural way. Then $S$ is a left and right denominator set in $\A$ and $\A_S \simeq \AAA$.
\end{theorem}
\begin{proof}
Let $A \in \A$ be non-zero and suppose that for some $s \in S$, $sA=0$. Without loss of generality we can assume that $A$ is homogeneous. Assume $A= \lambda \prod_{i=1}^k X_{i}^{m_i}$, $0 \neq \lambda \in D, m_i \in \N$. Let $B= \prod_{i=1}^k X_{k-i}^{-m_{k-i}}$. $AB \in D$, being the product of $\lambda$ and some $\sigma_i^{j}(a_i)$ for some $j \in \Z$, and each one of them is regular. Hence if $sA=0$, $sAB=0$, and so $\lambda=0$, which is absurd. So $s$ is regular in $D(a,\sigma)$. Also, as $S$ is commutative, $\sigma$-stable, and for each $d \in D$, $X_i^\pm d=\sigma^{\pm 1}(d) X_i$, it is clear that $S$ is a left and right Ore set in $D(a,\sigma)$. We have a natural map $\iota: \A \into \AAA$. As each element of $S$ becomes invertible in $\AAA_S$, by the universal property of localization we have a map $\theta:\A_S \into \AAA$. It is surjective because its image contains all the generators $X_i^\pm$. Hence $\theta$ is an isomorphism, and we are done.
\end{proof}

\begin{theorem}\label{GWA-pre-Galois-ring}
Let $\A=D(a,\sigma)$ be a degree $\alpha$ GWA which is regular and of surjective type. Let $S$ be the multiplicactively closed subset of $D$ generatedy by the $\sigma_j^m(a_i), i,j<\alpha, \, m \in \Z$.. $S$ is $\sigma$-stable. We have $\A_S \simeq D_S*\Za$.
\end{theorem}
\begin{proof}
    Consider the map $\A \into D_S *\Za$ given by $X_\beta^+ \mapsto \sigma_\beta$, $X_\beta^- \mapsto a_\beta \sigma_\beta^{-1}$. This is the desired isomorphism (cf. \cite[Proposition 13]{FS3}).
\end{proof}

\begin{corollary}
    The classical quotient ring $\Quot \, \A$ exists if and only if $\Quot \, D_S*\Za$ exists, in which case they are equal
\end{corollary}

This last corollary has a simple and nice application.

\begin{corollary}\label{birational}
    Let $\k$ be algebraically closed of zero characteristic, and let $\A$ be any degree $1$ GWA with base ring $\k[h]$. Then $\Frac \, A$ is either the field of fractions of the first Weyl algebra, the field of fractions of the quantum plane, or $\k(x,y)$
\end{corollary}
\begin{proof}
We must have $\sigma(h)=\lambda h +\beta$, for $\lambda \in \k^\times, \beta \in \k$. $\Frac \, A= \Frac \,\k(h)[y;\sigma]$. If $\lambda \neq 1$, the skew field of fractions if that of a quantum plane, If $\lambda=1$ and $\beta \neq 0$, then we have the field of fractions of the first Weyl algebra. If $\lambda=1$ and $\beta=0$, we have $\k(x,y)$. See \cite[pp. 30]{Dumas}.
\end{proof}

\section{Fixed rings of generalized Weyl algebras}

In this section we will show that fixed rings of generalized Weyl algebras under the action of certain finite groups are Galois orders.

\begin{proposition}\label{induced-actions-GWA}
    Let $D(a,\sigma)$ be a finite rank $n$ GWA, and $G$ a finite group of automorphisms of $D$ that, restricted to $\{a_1, \ldots, a_n \}$ permute the $a_i$. We will denote, given $g \in G$, by $g_a$ the corresponding permutation of $S_n$. Assume $g \circ\sigma_i \circ g^{-1} =\sigma_{g_a(i)}$. Then $G$ extend to a finite group of automorphisms of $D(a,\sigma)$: $g \in G$ sends $d \in D$ to $g(d)$ and $X_i^\pm$ to $X_{g_a(i)}^\pm$.
\end{proposition}

Let $g \in G$. Let $D \langle X, Y \rangle :=D \langle x_1, \ldots, x_n, y_1, \ldots, y_n \rangle$ be the free associative $D$-algebra and consider the morphism $\phi: D \langle X, Y \rangle \rightarrow D(a,\sigma)$ given by $x_i \mapsto X_{g_a(i)}^+, y_i \mapsto X_{g_a(i)}^-, d \mapsto g(d)$. Let us show that $\phi$ has the ideal of relations of $D(a,\sigma)$ (cf. Definition \ref{finite-GWA}) in the kernel, and hence that $\phi$ induces the desired automorphism of $D(a,\sigma)$. For instance, $\phi( X_i^+ d - \sigma_i(d) X_i^+)= X_{g_a(i)}^+ g(d) - g(\sigma_i(d))X_{g_a(i)}^+$, which is equal to $X_{g_a(i)}^+ g(d) - g(\sigma_i(g^{-1}(g(d))))X_{g_a(i)}^+=X_{g_a(i)}^+ g(d) - \sigma_{g_a(i)}(g(d)) X_{g_a(i)}^+=0$. Similarly, all other defining relations of $D(a,\sigma)$ lies in $\ker \phi$. Hence we are done.

First we recall a basic proposition.

\begin{proposition}\label{prop-basic-GWA}
    Let $D(a, \sigma)$ be a GWA of degree one. Then $D(a, \sigma)^{\otimes \, n}$ is a generalized Weyl algebra of degree $n$, denoted by $D^n(a^n, \sigma^n)$, where:

     $D^n= D^{\otimes \, n}$. 

     $a^n-(a_1,\ldots,a_n)$; $a_i= 1 \otimes \ldots \otimes 1 \otimes a \otimes 1 \otimes \ldots$, with $a$ in the $i$-th position.

     $\sigma^n=(\sigma_1,\ldots,\sigma_n)$; $\sigma_i = id \otimes \ldots \otimes id \otimes  \sigma \otimes id \otimes  \ldots$, with $\sigma$ in the $i$-th position.
\end{proposition}
\begin{proof}
    \cite{Bavula}.
\end{proof}

\textbf{\textit{From now on}} assume $\k$ algebraically closed of zero characteristic and $D$ an affine $\k$-algebra and commutative domain. In this way $D \otimes D$ is again a domain and affine, and hence a Noetherian, commutative algebra.

\begin{definition}
    A rank $n$ GWA $D^n(a^n,\sigma^n)$ is called \textit{tensorial} if it is the tensor product of $n$ copies of the same degree $1$ GWA $D(a,\sigma)$, with $\sigma$ an automorphism of infinite order. We also define $D^\omega(a^\omega, \sigma^\omega)=\varinjlim_{n \in \N} D^n(a^n, \sigma^n)$ (cf. Theorem \ref{GWA-as-limits}(b)).
\end{definition}

\begin{corollary}
    If $D(a,\sigma)$ is a degree 1 GWA from Example \ref{simple-GWA-examples}(a) or (b), $D^\omega(a^\omega,\sigma^\omega)$ is a simple algebra.
\end{corollary}
\begin{proof}
    It follows from Example \ref{simple-GWA-examples}(c) and Proposition \ref{prop-inductive-limit}(2)
\end{proof}

The study of the representation theory of the algebras $D^\omega(a^\omega, \sigma^\omega)$, along the lines of \cite{BBF} and \cite{FGM} seems to be an interesting task.

\begin{proposition}\label{Sn-action}
    Let $D^n(a^n,\sigma^n)$ be a degree n tensorial GWA. Then there is a natural action of the symmetric group $S_n$ on this algebra, induced from the action of $S_n$ on $D^n$ given by $\pi (d_1 \otimes \ldots \otimes d_n)=(d_{\pi(1)} \otimes \ldots \otimes d_{\pi(n)})$, $\pi \in S_n$. In this natural action on $D^n(a^n,\sigma^n)$, $\pi X_i^\pm=X_{\pi(i)}^\pm$.
\end{proposition}
\begin{proof}
    Immediate from Proposition \ref{induced-actions-GWA}.
\end{proof}

Let $F^n$ be $\Frac \, D^n$.

\begin{proposition}\label{GWA-principal}
    $D^n(a^n,\sigma^n)$ is a principal Galois order in $F^n*\Z^n$, where the canonical basis $\varepsilon_1 \ldots, \varepsilon_n$ is such that $\varepsilon_i$ acts on $F^n$ as $\sigma_i$.
\end{proposition}
\begin{proof}
    Like the proof of Theorem \ref{GWA-pre-Galois-ring}, $D^n(a^n,\sigma^n)$ embeds in $F^n*\Z^n$. By the form of the embedding, we have $D^n(a^n,\sigma^n)$ a Galois $D^n$-ring by Proposition \ref{prop-Hartwig-2}. It is clear that if $X \in D^n(a^n,\sigma^n)$, $X(D^n) \subset D^n$. So we have in fact a principal Galois order.
\end{proof}

The embedding $D^n(a^n,\sigma^n) \into F^n*\Z^n$ is $S_n$-equivariant (cf. Proposition \ref{Sn-action}), and $S_n$ acts on $\Z^n$ by conjugation. Hence we have a map $[D^n(a^n,\sigma^n)]^{S_n} \into (F^n*\Z^n)^{S_n}$. We write $D^*$ for $(D^n)^{S_n}$, and $F^*$ for its field of fractions.

\begin{proposition}\label{symmetric-GWA}
$[D^n(a^n,\sigma^n)]^{S_n}$ is a principal Galois $D^*$-order in $(F^n*\Z^n)^{S_n}$ 
\end{proposition}
\begin{proof}
    By Theorem \ref{MS} $[D^n(a^n,\sigma^n)]^{S_n}$ is an affine $\k$-algebra. It contains $\sum X_i^+$ and $\sum X_i^-$. These elements have support that generates the monoid $\Z^n$. Hence by Proposition \ref{prop-Hartwig-2} we have a Galois ring; by the previous Proposition and by Proposition \ref{main-prop}, we have a principal Galois order.
\end{proof}

Before proceeding, we well recall an important result form \cite[Section 2.7]{JW}. For $m \geq 1$, let $G_m \subset k^\times$ be the cyclic group consisting of $m$-th roots of unities. Let $D(a,\sigma)$ be a degree 1 GWA. $G_m$ acts on it by fixing $D$ and, if $\zeta \in G_m$, $\zeta.X^+=\zeta X^+$, and $\zeta ,X^-=\zeta^{-1} X^-$.

\begin{theorem}\label{JW}
$D(a,\sigma)^{G_m}$, with the above action, is again a generalized Weyl algebra, namely $D(a_m,\sigma_m)$, where $a_m = \prod_{i=0}^{m-1} \sigma^{-1}(a)$, and $\sigma_m=\sigma^m$.    
\end{theorem}

We can also make $G_m^{\otimes n}$ act diagonally on $D^n(a^n, \sigma^n)$. Extending the previous result, it is easy to see that the fixed ring under the action of this finite group will be $D^n(a_m^n,\sigma_m^n)$. 

Now lets recall the definition of the groups $G(m,p.n)$, $p|m$. Define $A(m,p,n)$ to be the subgroup of $G_m^{\otimes n}$ consisting of $(\zeta_1,\ldots,\zeta_n)$ such that $(\prod \zeta_i)^{m/p}=1$. $G(m,p.n)$ is then the semidirect product $A(m,p,n) \rtimes S_n$, where $S_n$ permutes the entries of $A(m,p,n)$.

We first consider the case $p=1$, i.e., $G(m,1,n)=G_m^{\otimes n} \rtimes S_n$. Given any ring $R$, $G$ a group of automorphisms of $R$ and $N$ a normal subgroup, we have $R^G=(R^N)^{G/N}$.

Hence $D^n(a^n.\sigma^n)^{G(m,1,n)}$ is $D^n(a_m^n,\sigma_m^n)^{S_n}$. Combining Propositions \ref{GWA-principal} and \ref{symmetric-GWA} we immediatly have:

\begin{theorem}\label{like-FS3}
$D^n(a^n.\sigma^n)^{G(m,1,n)}$ is a principal Galois $D^*$-ring in $(F^*\Z^n)^{S_n}$. 
\end{theorem}

In the next lemma assume $\operatorname{char} \, \k=0$. I has some independent interest.

\begin{lemma}\label{new-cool-lemma}
    If $U$ is an associative affine and Noetherian algebra and $G$ a finite group of automorphisms of $U$ such that $U^G$ is a Galois $\Gamma$-ring in $\K$, and $H$ is a subgroup of $G$, then $U^H$ is also a Galois $\Gamma$-ring in $\K$.
\end{lemma}
\begin{proof}
    $U^G$ is a finitely generated $\Gamma$-algebra. If $u_1, \ldots, u_k$ are elements of $U^G$ that generate $U^G$, by Proposition \ref{prop-Hartwig-2} we can assume that $\bigcup \supp \, u_i$ generates $\M$ as a monoid. Now, $U^G \subset U^H$. By Theorem \ref{MS}, $U^H$ is still finitely generated and contain the elements $u_1, \ldots, u_k$. Hence it has a finite set of generators whose support generates $\M$ as a monoid. Hence, by another application of Proposition \ref{prop-Hartwig-2}, we conclude that $U^H$ is a Galois $\Gamma$-ring in $\K$.
\end{proof}

The next Lemma has the same proof as \cite[Proposition 30]{FS3}:

\begin{lemma}
    $D^n(a^n,\sigma^n)^{G(m,p,n)} = \bigoplus_{k=0}^{p-1} (X_1^+X_2^+ \ldots X_n^+)^{mp/k} D^n(a^n,\sigma^n)^{G(m,1,n)}$.
\end{lemma}

The following is the main theorem of this section, which generalizes the main result of \cite{FS3} for \textit{all $p$} and \textit{all generalized Weyl algebras}:

\begin{theorem}\label{main-fixed-GWA}
$D^n(a^n,\sigma^n)^{G(m,p,n)}$ is a principal Galois $D^*$-order in $F^n*\Z^n$ where each element of the basis of $\Z^n$, $\varepsilon_i$, acts by $\sigma_i^m$
\end{theorem}
\begin{proof}
    The Galois ring realization is clear by Theorem \ref{JW}, Theorem \ref{like-FS3} and Lemma \ref{new-cool-lemma}. $D^n(a^n,\sigma^n)^{G(m,p,n)}(D^*)=\bigoplus_{k=0}^{p-1} (X_1^+X_2^+ \ldots X_n^+)^{mp/k} D^n(a^n,\sigma^n)^{G(m,1,n)}(D^*) \subset D^*$. So it is also a principal Galois order.
\end{proof}

Now we discuss the Gelfand-Kirillov Hypothesis

\begin{theorem}\label{GK-2}
    If $D(a,\sigma)$ is a GWA of rank I and classical type (cf. Example \ref{simple-GWA-examples}), then $\Frac \, (D^(a^,\sigma^n))^{G(m,p,n)} \simeq \Frac W_n(\k)$. In case it is a GWA of rank I of quantum type, $\Frac \, (D^n(a^,\sigma^n)^{G(m,p,n)} \simeq \Frac(\k_{q^{m/p}}[x,y] \otimes \k_q[x,y]^{\otimes n-1})$, where $\k_q[x,y]$ is the quantum plane
\end{theorem}
\begin{proof}
    In the first case, we have the skew field of fractions of $(\k(t_1,\ldots,t_n)*\Z^n)^{S_n}$, which is known to be $\Frac \, A_n(\k)$ (cf. \cite{FMO}, \cite{FS}). In the second case, it follows from \cite[Theorem 1.1]{H0}
\end{proof}

We finish this section by showing that if $D$ is a polynomial algebra or a ring of Laurent polynomials (in one indeterminate), then $D^n(a^n,\sigma^n)^{G(m,p,n)}$ is free over its Harish-Chandra subalgebra $\k[h_1,\ldots,h_n]^{S_n}$ (or $\k[h_1^{\pm 1}, \ldots, h_n^{\pm 1}]^{S_n}$).

For this we will need the following important result of H. Bass:

\begin{proposition}\label{Bass}
    Let $A$ be a commutative Noetherian ring with no non-trivial idempotents. If $P$ is a non-finitely generated $A$-module, then $P$ is, in fact, free.
\end{proposition}
\begin{proof}
    \cite[Corollary 4.5]{Bass}.
\end{proof}

It is well known that $\k[h_1,\ldots,h_n]$ is a free $\k[h_1, \ldots, h_n]^{S_n}$-module of rank $n!$, and a similar result holds for any pseudo-reflection group by Chevalley-Shephard-Todd Theorem, as is well known.

\begin{lemma}\label{Laurent}
$\k[h_1^{\pm 1}, \ldots, h_n^{\pm 1}]$ is a free $\k[h_1^{\pm 1}, \ldots, h_n^{\pm 1}]^{S_n}$-module of rank $n!$.  
\end{lemma}
\begin{proof}
Let $e=\prod h_i$, and $E$ the multiplicative closed set generated by $e$. $\k[h_1^{\pm 1}, \ldots, h_n^{\pm 1}]\simeq \k[h_1,\ldots,h_n]_E$, and as $e$ is $S_n$-invariant,  $\k[h_1^{\pm 1}, \ldots, h_n^{\pm 1}]^{S_n}=\k[h_1,\ldots,h_n]^{S_n}_E$. As localization is an exact functor and $\k[h_1,\ldots,h_n]$ is $E$-torsion free, from the isomorphism $\k[h_1,\ldots,h_n] \simeq \oplus_{n!} \k[h_1,\ldots,h_n]^{S_n}$, we get $\k[h_1^{\pm 1}, \ldots, h_n^{\pm 1}] \simeq \oplus_{n!} \k[h_1^{\pm 1}, \ldots, h_n^{\pm 1}]^{S_n}$. This finishes the proof.
\end{proof}

\begin{theorem}
    If $D=\k[h]$ or $\k[h^{\pm 1}]$, then $D^n(a^n,\sigma^n)^{G(m,p,n)}$ is a free module over its Harish-Chandra subalgebra $D^*$.
\end{theorem}
\begin{proof}
    By Chevalley-Shephard-Todd Theorem or by the previous Lemma, $D^n$ is a free $D^*$-module of rank $n!$, and we already saw that $D^n(a^n,\sigma^n)$ is a free $D^n$-module, not finitely generated. Them by Bass result Proposition \ref{Bass} and Proposition \ref{main-prop}(2), $D^n(a^n,\sigma^n)^{G(m,p,n)}$ is a free $D^*$-module.
\end{proof}

\begin{corollary}
    In the conditions of the Theorem above, let $\m \in \SpecMax \, D^*$, and suppose $|\M_\m|$ finite. If $\mathsf{M} \in \Irred(\m)$, then $\dim \mathsf{M}^\m \leq n! |\M_\m|$.
\end{corollary}
\begin{proof}
    This is a combination of Lemma \ref{G-T-1} and \cite[Lemma 4.12(c)]{FO2}.
\end{proof}

In the case of the Weyl algebra, the bound can be improved to $n!$. We conjecture that the upper bound is actually $1$.

\section*{Acknowledgments}
I acknowledge the fundamental contribution of coffee to the writing of this manuscript, which was carried out mostly between 2 and 6 AM in the morning. I also thank my wife Pan for her neverending love.

\end{document}